\documentclass[final,12pt]{elsarticle}

\usepackage{pifont,natbib,geometry}

\usepackage{amssymb,amsmath,amsfonts,amsthm,amscd,stmaryrd}

\theoremstyle{plain}
\newtheorem{thm}{Theorem}

\newtheorem{lem}[thm]{Lemma}
\newtheorem{prop}[thm]{Proposition}

\theoremstyle{definition}
\newtheorem{definition}[thm]{Definition}
\newtheorem{rem}[thm]{Remark}
\newtheorem{exam}[thm]{Example}

\numberwithin{thm}{section}
\numberwithin{equation}{section}

\newcommand{\EQ}[1]{\eqref{eq:#1}} 
\newcommand{\LEM}[1]{Lemma~\ref{lem:#1}}    
    
\newcommand{\THM}[1]{Theorem~\ref{thm:#1}}  
\newcommand{\REM}[1]{Remark~\ref{rem:#1}}  
\newcommand{\PROP}[1]{Proposition~\ref{prop:#1}}

\newcommand{\SEC}[1]{Section~\ref{sec:#1}}  
\newcommand{\EXAM}[1]{Example~\ref{exam:#1}}
\newcommand{\APP}[1]{Appendix~\ref{app:#1}}

\newcounter{hypo}

\newcommand{\HYP}[1]{\ref{hyp:#1}}
\renewcommand{\labelenumi}{(\roman{enumi})}

\DeclareMathOperator{\trace}{trace}

\DeclareMathOperator{\divg}{div}

\newcommand{\Sy}{\ensuremath{\mathbb{S}^n}}
\newcommand{\consp}{\delta_1}
\newcommand{\consz}{\delta_0}
\newcommand{\puccisub}[2]{\mathcal{P}^-_{#1,#2}}
\newcommand{\Puccisub}[2]{\mathcal{P}^+_{#1,#2}}
\newcommand{\pucci}{\mathcal{P}^-}
\newcommand{\Pucci}{\mathcal{P}^+}
\newcommand{\inOmega}{\quad \mbox{in} \quad \Omega}

\newcommand{\ep}{\ensuremath{\varepsilon}}
\newcommand{\R}{\ensuremath{\mathbb{R}}}

\newcommand{\N}{\ensuremath{\mathbb{N}}}

\newcommand{\consABP}{C_0}

\newcommand{\consEESD}{C}

\newcommand{\imat}{I_n}

\newcommand{\radonmeas}{\ensuremath{\mathcal{R}(\bar{\Omega})}} 

\newcommand{\probmeas}{\ensuremath{\mathcal{M} ( \bar{\Omega} )     }} 

\newcommand{\dvmeas}{\ensuremath{\mathcal{V}(F,\Omega)}} 
\newcommand{\dvmeascust}[2]{\ensuremath{\mathcal{V}(#1,#2)}}

\newcommand{\posfun}{\ensuremath{C^2_+({\bar{\Omega}})}}
\newcommand{\weakly}{\ensuremath{\rightharpoonup}}
\newcommand{\efunpos}{\ensuremath{\varphi^+_{1}}}
\newcommand{\efunadjpos}[1]{\ensuremath{\varphi^*_{#1}}}

\journal{Journal of Differential Equations}

\begin{document}

\begin{frontmatter}

\title{The Dirichlet problem for the Bellman equation at resonance}
\author{Scott N. Armstrong}
\address{Department of Mathematics,
University of California, Berkeley, CA 94720.}
\ead{sarm@math.berkeley.edu}
\ead[url]{http://math.berkeley.edu/\~{}sarm}
\date{\today}

\begin{abstract}
We generalize the Donsker-Varadhan minimax formula for the principal eigenvalue of a uniformly elliptic operator in nondivergence form to the first principal half-eigenvalue of a fully nonlinear operator which is concave (or convex) and positively homogeneous. Examples of such operators include the Bellman operator and the Pucci extremal operators. In the case that the two principal half-eigenvalues are not equal, we show that the measures which achieve the minimum in this formula provide a partial characterization of the solvability of the corresponding Dirichlet problem at resonance.
\end{abstract}


\end{frontmatter}


\section{Introduction} \label{sec:introduction}

Consider a nondivergence form uniformly elliptic operator
\begin{equation}\label{eq:linearoperator}
L = -a^{ij}(x)\partial_i\partial_j + b^j(x) \partial_j + c(x)
\end{equation}
in a smooth, bounded domain $\Omega \subseteq \R^n$. The celebrated minimax formula of Donsker and Varadhan \cite{Donsker:1975,Donsker:1976} states that the principal eigenvalue $\lambda_1(L,\Omega)$ of $L$ can be expressed by the minimax formula
\begin{equation}\label{eq:DV-minimax}
\lambda_1(L,\Omega) = \min_{\mu \in \probmeas} \sup_{u\in C^2_+(\bar{\Omega})} \int_\Omega \left( \frac{Lu}{u} \right)\! (x)\, d\mu(x).
\end{equation}
Here $\probmeas$ denotes the space of Borel probability measures on $\bar{\Omega}$, and $C^2_+(\bar{\Omega})$ is the set of positive $C^2$ functions on $\bar{\Omega}$. The minimum in \EQ{DV-minimax} is achieved by a unique probability measure $\mu$, given by
\begin{equation*}
d\mu(x) = \varphi_1(x) \varphi^*_1(x)\, dx,
\end{equation*}
where $\varphi_1$ is the principal eigenfunction of $L$, and $\varphi^*_1$ is the principal eigenfunction of the adjoint operator $L^*$, normalized according to
\begin{equation*}
\int_\Omega \varphi_1(x) \varphi^*_1(x) \, dx = 1.
\end{equation*}
Consequently, we may characterize the solvability of the boundary value problem
\begin{equation}\label{eq:DP-intro}
\left\{ \begin{aligned}
& L u = \lambda_1(L,\Omega) u + f & \mbox{in} & \ \Omega, \\
& u = 0 & \mbox{on} & \ \partial \Omega,
\end{aligned}\right.
\end{equation}
in terms of the measure $\mu$. According to the Fredholm alternative, there exists a solution of the problem \EQ{DP-intro} if and only if
\begin{equation}\label{eq:solvability-criterion}
0 = \int_\Omega f(x) \varphi^*_1(x) \, dx=  \int_\Omega \frac{f}{\varphi_1}\, d\mu.
\end{equation}

\medskip

Let $\{ L^k \}_{k \in A}$ be a family of linear, uniformly elliptic operators in nondivergence form. Lions \cite{Lions:1983d} was the first to study the principal half-eigenvalues of the Hamilton-Jacobi-Bellman operator
\begin{equation}\label{eq:bellman}
H(D^2u,Du,u,x) : = \inf_{k \in A} \left\{ L^k u \right\}.
\end{equation}
He demonstrated the existence of an eigenvalue $\lambda^+_1(H,\Omega)$ corresponding to a positive eigenfunction $\varphi^+_1 > 0$ and another eigenvalue $\lambda^-_1(H,\Omega)$ corresponding to a negative eigenfunction $\varphi^-_1 < 0$. The numbers $\lambda^\pm_1(H,\Omega)$ are called \emph{principal eigenvalues}, \emph{half-eigenvalues}, or \emph{demi-eigenvalues} of the operator $H$. While in general $\lambda^+_1(H,\Omega) \neq \lambda^-_1(H,\Omega)$, the concavity of $H$ ensures that the inequality $\lambda^+_1(H,\Omega) \leq \lambda^-_1(H,\Omega)$ is satisfied.

The principal half-eigenvalues have many properties analogous to the principal eigenvalue of a linear operator. The operator $H $ satisfies the comparison principle in the domain $\Omega$ if and only if $\lambda^+_1(H,\Omega) > 0$.  The Dirichlet problem
\begin{equation} \label{eq:DP-intro-2}
\left\{ \begin{aligned}
& H(D^2u,Du,u,x) = \lambda u + f & \mbox{in} & \ \Omega,\\
& u = 0 & \mbox{on} & \ \partial \Omega,
\end{aligned} \right.
\end{equation}
has a unique solution $u\in W^{2,\infty}(\Omega)$ for any given smooth function $f$ if $\lambda < \lambda^+_1(H,\Omega)$. Moreover, \EQ{DP-intro-2} has a unique nonpositive solution for any given smooth nonpositive $f$ if and only if $\lambda < \lambda^-_1(H,\Omega)$. These solutions can be expressed as the value functions of an optimal stochastic control problem associated with the operators $L^k$.

In \cite{Lions:1983d}, these facts were proved using mostly stochastic methods. Recently, several authors \cite{Birindelli:2006,Birindelli:2007,Ishii:preprint, Quaas:2008} employed PDE methods to generalize the results of \cite{Lions:1983d} to more general elliptic operators. Ishii and Yoshimura \cite{Ishii:preprint} have shown that an operator $F=F(D^2u,Du,u,x)$ possesses two principal half-eigenvalues assuming only that $F$ is uniformly elliptic and positively homogeneous jointly in its first three arguments (see \THM{eigenvalues}), and established related comparison principles and existence for the Dirichlet problem. Similar results have been independently obtained by Birindelli and Demengel \cite{Birindelli:2006,Birindelli:2007}, who considered degenerate, singular operators like the $p$-Laplacian, and Quaas and Sirakov \cite{Quaas:2008}, who studied nonlinear operators which are also convex or concave in $u$, but may be only measurable in $x$.

\medskip

In this paper, we will generalize the minimax formula \EQ{DV-minimax} to the first principal eigenvalue of an operator that is concave (or convex), as well as uniformly elliptic and positively homogeneous. Motivated by the solvability condition \EQ{solvability-criterion}, we will show that the probability measures(s) attaining the minimum provide a partial answer to the question of existence of solutions to the corresponding Dirichlet problem.

\begin{thm}\label{thm:minimax-eigenvalue}
Suppose that $\Omega \subseteq \R^n$ is a smooth bounded domain, and $F$ is a nonlinear operator satisfying \HYP{Fcontinuous}, \HYP{Felliptic}, \HYP{Fhomogeneous}, and \HYP{Fconcave}, below. Then the principal half-eigenvalue $\lambda^+_1(F,\Omega)$ satisfies the minimax formula
\begin{equation}\label{eq:minimax-1}
\lambda^+_1(F,\Omega) = \min_{\mu \in \probmeas} \sup_{\varphi\in \posfun} \int_\Omega \frac{F(D^2\varphi(x),D\varphi(x),\varphi(x),x)}{\varphi(x)}\, d\mu(x).
\end{equation}
Moreover, for each probability measure $\mu \in \probmeas$ which attains the minimum in \EQ{minimax-1}, there exists a function $\varphi^*_\mu \in L^{n/(n-1)}(\Omega)$ such that $\varphi^*_\mu > 0$ a.e. in $\Omega$ and $d\mu = \varphi^*_\mu \varphi^+_1\, dx$.
\end{thm}

Let $\dvmeas \subseteq \probmeas$ denote the subset of probability measures attaining the minimum in \EQ{minimax-1}. In contrast to the situation for linear operators, $\dvmeas$ is not a singleton set in general (see \EXAM{nonunique-measures} below). Our next result is a partial characterization of the solvability of the Dirichlet problem
\begin{equation} \label{eq:DP}
\left\{ \begin{aligned}
& F(D^2u,Du,u,x) =  \lambda^+_1(F,\Omega) u + f & \mbox{in} & \ \Omega, \\
& u = 0 & \mbox{on} & \ \partial \Omega,
\end{aligned}\right.
\end{equation}
expressed in terms of $\dvmeas$.

\begin{thm}\label{thm:Dirichlet-at-resonance}
Let $\Omega$ and $F$ be as in \THM{minimax-eigenvalue}, and $f\in C(\Omega) \cap L^n(\Omega)$. Then the inequality \begin{equation} \label{eq:necessary}
\max_{\mu \in \dvmeas} \int_\Omega f \varphi^*_\mu \, dx \leq 0
\end{equation}
is necessary for the solvability of the Dirichlet problem \EQ{DP}. Suppose in addition that
\begin{equation}\label{eq:plus-less-minus}
\lambda^+_1(F,\Omega) < \lambda^-_1(F,\Omega).
\end{equation}
Then the strict inequality
\begin{equation} \label{eq:sufficient}
\max_{\mu \in \dvmeas} \int_\Omega f \varphi^*_\mu \, dx < 0
\end{equation}
is sufficient for the solvability of the Dirichlet problem \EQ{DP}.
\end{thm}

The hypothesis \EQ{plus-less-minus} is relatively generic for a nonlinear operator. For example, the Bellman operator $H$ given by \EQ{bellman} fails to satisfy \EQ{plus-less-minus} only if the linear operators belonging to the family $\{ L^k \}_{k\in A}$ share the same principal eigenvalue \emph{and} principal eigenfunction. See \cite[Remark II.6]{Lions:1983d} or \cite[Example 3.11]{Armstrong:2009}.

In the borderline case that
\begin{equation*}
\max_{\mu \in \dvmeas} \int_\Omega f\varphi^*_\mu\,dx =0,
\end{equation*}
solutions of \EQ{DP} may or may not exist; see \EXAM{nssn} below. We also wish to mention the work of Sirakov \cite{Sirakov:preprint}, who has studied the existence and nonuniqueness of solutions of \EQ{DP} in the case that $\lambda^+_1(F,\Omega) < \lambda < \lambda^-_1(F,\Omega)$.

\medskip

Soon after submitting this article for publication, I became aware of the recent, very interesting work of Felmer, Quaas, and Sirakov \cite{FQS:preprint}, which contains results similar to \THM{Dirichlet-at-resonance}. In particular, using topological techniques, they have shown that under assumption \EQ{plus-less-minus}, for each fixed function $h\in L^n(\Omega)$, there exists a number $t^* = t^*(h)$ such that the Dirichlet problem \EQ{DP} has a solution for $f := h - t \varphi^+_1$ provided that $t > t^*$, and has no solution if $t < t^*$. Using \THM{Dirichlet-at-resonance}, we see that $t^*$ can be expressed by
\begin{equation}\label{eq:tstar}
t^*(h) = \max_{\mu \in \dvmeas} \int_\Omega h \varphi^*_\mu \, dx.
\end{equation}
Moreover, in \cite{FQS:preprint} there is much information about the set of solutions along the curves $f_t = h - t\varphi^+_1$. In particular, results in \cite{FQS:preprint} together with \EQ{tstar} imply that if \EQ{plus-less-minus} and \EQ{sufficient} hold, a solution of \EQ{DP} is unique. The preprint \cite{FQS:preprint} also contains results for the Dirichlet problem in the case that $\lambda^+_1(F,\Omega) < \lambda \leq \lambda^-_1(F,\Omega)$, and to which my methods do not obviously apply.

\medskip

This paper is organized as follows. In \SEC{preliminaries} we state our hypotheses and review some known results for the principal half-eigenvalues of fully nonlinear operators. \SEC{mainresults} contains the proofs of \THM{minimax-eigenvalue} and \THM{Dirichlet-at-resonance}. In \SEC{examples} we give a sufficient condition for the set $\dvmeas$ of minimizing measures to be a singleton set, and study a couple of simple examples.

\medskip

This article was completed while I was Ph.D. student at the University of California, Berkeley. I wish to thank my advisor Lawrence C. Evans and the Department of Mathematics for their continual guidance and support. I would also like to acknowledge several interesting conversations with Isabeau Birindelli on this topic, and to thank Boyan Sirakov for sending me the preprint \cite{FQS:preprint}.


\section{Preliminaries}

\label{sec:preliminaries}

\subsection{Notation and Hypotheses}

Throughout this paper, we take $\Omega$ to be a bounded, smooth, and connected open subset of $\R^n$. Let $\Sy$ denote the set of $n$-by-$n$ real symmetric matrices. We denote the space of Radon measures on $\bar{\Omega}$ by $\radonmeas$, and the set of probability measures on $\bar{\Omega}$ by
\begin{equation*}
\probmeas = \left\{ \mu \in \radonmeas : \mu \geq 0 \ \mbox{and} \ \mu(\bar{\Omega}) = 1 \right\}.
\end{equation*}
For any $k \in \N$, set
\begin{equation*}
C^k_+(\bar{\Omega}) = \left\{ u \in C^k(\bar{\Omega}) : u > 0 \ \mbox{on} \ \bar{\Omega} \right\}.
\end{equation*}
For $r,s\in \R$, define $r \wedge s = \min\{r,s\}$ and $r \vee s = \max \{r,s\}$. If $r\in \R$, then we write $r^+ = r \vee 0$ and $r^- = - (r\wedge 0)$. For $M \in \Sy$ and $0 < \gamma \leq \Gamma$, define the uniformly elliptic operators
\begin{equation*}
\Puccisub{\gamma}{\Gamma} (M) = \sup_{A\in \llbracket\gamma,\Gamma\rrbracket} \left[ - \trace(AM) \right] \quad \mbox{and} \quad \puccisub{\gamma}{\Gamma} (M) = \inf_{A\in \llbracket\gamma,\Gamma\rrbracket} \left[ - \trace(AM) \right],
\end{equation*}
where the set $\llbracket\gamma,\Gamma\rrbracket \subseteq \Sy$ consists of the symmetric matrices the eigenvalues of which lie in the interval $[ \gamma, \Gamma]$. The nonlinear operators $\Puccisub{\gamma}{\Gamma}$ and $\puccisub{\gamma}{\Gamma}$ are the Pucci extremal operators.

\medskip

We impose the following standing assumptions. Our nonlinear operator $F$ is a function
\begin{equation*}
F:\Sy \times \R^n \times \R \times \bar{\Omega} \rightarrow \R
\end{equation*}
satisfying the following:
\renewcommand{\labelenumi}{(F\arabic{hypo})}
\begin{enumerate}
\usecounter{hypo}
\item For each $K > 0$, there exists a constant $B_K$ and a positive constant $\frac{1}{2}< \nu \leq 1$, depending on $K$, such that
\begin{equation*}
|F(M,p,z,x) - F(M,p,z,y) | \leq B_K |x-y|^\nu (|M|+1)
\end{equation*}
for all $M\in \Sy$, $p\in \R^n$, $z\in \R$, and $x,y\in \bar{\Omega}$ satisfying $|p|, |z| \leq K$. \label{hyp:Fcontinuous}
\item There exist constants $\consp,\consz>0$ and $0<\gamma \leq \Gamma$ such that
\begin{multline*}
\qquad \quad \puccisub{\gamma}{\Gamma}(M-N) - \consp |p-q| - \consz |z-w| \leq F(M,p,z,x) - F(N,q,w,x) \\ 
\leq \Puccisub{\gamma}{\Gamma}(M-N) + \consp |p-q| + \consz|z-w|
\end{multline*}
for all $M,N\in \Sy$, $p,q\in R^n$, $z,w \in \R$, $x\in \bar{\Omega}$.\label{hyp:Felliptic}
\item $F$ is positively homogeneous of order one, jointly in its first three arguments; i.e.,
\begin{equation*}
F(tM,tp,tz,x) = tF(M,p,z,x)\quad \mbox{for all} \quad t\geq 0
\end{equation*}
and all $M\in \Sy$, $p\in \R^n$, $z\in \R$, $x\in \bar{\Omega}$.\label{hyp:Fhomogeneous}
\item $F$ is concave jointly in its first three arguments; i.e., the map
\begin{equation*}
(M,p,z) \mapsto F(M,p,z,x) \quad \mbox{is concave}
\end{equation*}
for every $x \in \bar{\Omega}$.
\label{hyp:Fconcave}
\end{enumerate}
\renewcommand{\labelenumi}{(\roman{enumi})} 

\medskip

Hypothesis \HYP{Felliptic} implies $F$ is uniformly elliptic in the familiar sense that
\begin{equation*}
- \gamma \trace (N) \geq F(M+N,p,z,x) - F(M,p,z,x)
\end{equation*}
for every nonnegative definite matrix $N\in \Sy$. (In particular, we adopt the sign convention that regards $-\Delta$ as elliptic.)

Recall that a positively homogeneous function is concave if and only if it is superlinear. Thus \HYP{Fhomogeneous} and \HYP{Fconcave} imply that
\begin{equation}\label{eq:Fsuperlinear}
F(M+N,p+q,z+w,x) \geq F(M,p,z,x) + F(N,q,w,x)
\end{equation}
for all $M,n\in \Sy$, $p,q\in \R^n$, $z,w\in\R$ and $x\in\bar{\Omega}$. We may rewrite this as
\begin{equation} \label{eq:Fsuperlinearminus}
F(M-N,p-q,z-w,x) \leq F(M,p,z,x) - F(N,q,w,x),
\end{equation}
and in particular we have 
\begin{equation}\label{eq:Fsuperlinearneg}
F(M,p,z,x) \leq - F(-M,-p,-z,x).
\end{equation}
We may replace concavity with convexity in the hypothesis \HYP{Fconcave} and our results with still hold, provided that we make appropriate sign changes in our statements. This follows from the simple observation that if $G$ satisfies \HYP{Fcontinuous}-\HYP{Fhomogeneous} but is convex in $(M,p,z)$, then the operator
\begin{equation*}
\tilde{G}(M,p,z,x) := -G(-M,-p,-z,x)
\end{equation*}
satisfies \HYP{Fcontinuous}-\HYP{Fconcave}.
\medskip

All differential equations and inequalities, unless otherwise indicated, are assumed to be satisfied in the viscosity sense. See \cite{UsersGuide,Fleming:Book} for definitions and an introduction to the theory of viscosity solutions of second order elliptic equations.

We will make use of $W^{2,p}$ and $C^{2,\alpha}$ estimates for viscosity solutions of concave, second-order uniformly elliptic equations. See \cite{Caffarelli:Book,Winter:2008} for details.

\subsection{Principal half-eigenvalues}

We now review some known facts regarding principal eigenvalues of $F$. The results in this subsection were substantively reported in \cite{Lions:1983d}, and have recently been generalized in \cite{Armstrong:2009, Ishii:preprint,Quaas:2008}. While we continue to assume that our nonlinear operator $F$ is concave, most of our conclusions hold for operators satisfying only \HYP{Fcontinuous}-\HYP{Fhomogeneous}; see \cite{Armstrong:2009,Ishii:preprint} for details.

\medskip

The following comparison principle is essential to the theory of principal eigenvalues of nonlinear operators. It is based on an insight that goes back to the work of Berestycki, Nirenberg, and Varadhan \cite{Berestycki:1994}.

\begin{thm}[Comparison principle for positively homogeneous operators]\label{thm:BNV}
Suppose $u,v\in C(\bar{\Omega})$ and $f\in C(\Omega)$ satisfy
\begin{equation}\label{eq:teh}
F(D^2u, Du, u, x) \leq f \leq F(D^2v,Dv,v,x) \quad \mbox{in} \ \Omega,
\end{equation}
and that one of the following conditions holds:
\begin{enumerate}
\item \label{enum:teh-1}
$f \leq 0$ and $u < 0$ in $\Omega$, $v \geq 0$ on $\partial \Omega$, \ \mbox{or}
\item \label{enum:teh-2}
$f \geq 0$ and $v > 0$ in $\Omega$, $u \leq 0$ on $\partial \Omega$.
\end{enumerate}
Then either $u \leq v$ in $\Omega$, or $v\equiv tu$ for some positive constant $t\neq 1$.
\end{thm}
See \cite{Armstrong:2009} for a proof of \THM{BNV}.

\begin{thm}[See Ishii and Yoshimura \cite{Ishii:preprint} or Quaas and Sirakov \cite{Quaas:2008}]\label{thm:eigenvalues}
There exist functions $\varphi^+_1,\varphi^-_1 \in C^{2,\alpha}(\Omega)$ such that $\varphi^+_1 > 0$ and $\varphi^-_1<0$ in $\Omega$, and which satisfy
\begin{equation}\label{eq:Feigenvalues}
\left\{ \begin{aligned}
& F(D^2\varphi_1^+,D\varphi_1^+,\varphi_1^+,x) = \lambda_1^+(F,\Omega) \varphi_1^+ & \mbox{in} & \ \Omega, \\
& F(D^2\varphi_1^-,D\varphi_1^-,\varphi_1^-,x) = \lambda_1^-(F,\Omega) \varphi_1^- & \mbox{in} & \ \Omega, \\
& \varphi_1^+ = \varphi^-_1 = 0 & \mbox{on} & \ \partial \Omega. 
\end{aligned}\right.
\end{equation}
Moreover, the eigenvalue $\lambda_1^+(F,\Omega)$ ($\lambda_1^-(F,\Omega)$) is unique in the sense that if $\rho$ is another eigenvalue of $F$ in $\Omega$ associated with a nonnegative (nonpositive) eigenfunction, then $\rho=\lambda_1^+(F,\Omega)$ ($\rho=\lambda_1^-(F,\Omega)$); and is simple in the sense that if $\varphi \in C(\bar{\Omega})$ is a solution of \EQ{Feigenvalues} with $\varphi$ in place of $\varphi^+_1$ ($\varphi^-_1$), then $\varphi$ is a constant multiple of $\varphi_1^+$ ($\varphi_1^-$).
\end{thm}

The principal eigenvalues $\lambda^\pm_1(F,\Omega)$ satisfy the maximin formulas
\begin{equation}\label{eq:maximin-formula-2}
\lambda^+_1(F,\Omega) = \sup_{\varphi\in C_+^2(\Omega)} \inf_{x\in \Omega} \frac{F(D^2\varphi(x),D\varphi(x),\varphi(x),x)}{\varphi(x)},
\end{equation}
and
\begin{equation}\label{eq:maximin-formula-3}
\lambda^-_1(F,\Omega) = \sup_{ \varphi\in C^2_+(\Omega)} \inf_{x\in \Omega} \frac{-F(-D^2\varphi(x),-D\varphi(x),-\varphi(x),x)}{\varphi(x)}.
\end{equation}
Along with \EQ{Fsuperlinearneg}, these imply that
\begin{equation*}
\lambda^+_1(F,\Omega) \leq \lambda^-_1(F,\Omega).
\end{equation*}
If $\Omega' \subsetneq \Omega$, then we can compare the principal eigenfunctions of $F$ on the domains $\Omega'$ and $\Omega$ and employ \THM{BNV} to immediately conclude that
\begin{equation}\label{eq:eigenvalues-monotonic-in-domain}
\lambda^\pm_1(F,\Omega) < \lambda^\pm(F,\Omega').
\end{equation}

We will need the following extension of the Alexandroff-Bakelman-Pucci inequality, which was essentially proven in the convex case by Quaas and Sirakov \cite{Quaas:2008}. 

\begin{thm}[ABP Inequality] \label{thm:ABP-eigen}
There exists a constant $C_1$, depending only only on $\Omega$, $n$, $\gamma$, $\Gamma$, and $\consp$, such that for any $\lambda < \lambda^+_1(F,\Omega)$, $f\in C(\Omega) \cap L^n(\Omega)$, and any subsolution $u\in C(\bar{\Omega})$ of
\begin{equation*}
F(D^2u,Du,u,x) \leq \lambda u + f \quad \mbox{in} \quad \{ u > 0 \},
\end{equation*}
we have the estimate
\begin{equation}\label{eq:ABP-eigen}
\sup_\Omega u^+ \leq C_1 \left( 1 + ( \lambda^+_1(F,\Omega) - \lambda )^{-1} \right) \left( \sup_{\partial \Omega} u^+ + \|f^+ \|_{L^n(\Omega)} \right).
\end{equation}
Likewise, for any $\lambda < \lambda^-_1(F,\Omega)$, $f\in C(\Omega) \cap L^n(\Omega)$, and any supersolution $u\in C(\bar{\Omega})$ of
\begin{equation*}
F(D^2u,Du,u,x) \geq \lambda u + f \quad \mbox{in} \quad \{ u < 0 \},
\end{equation*}
we have the estimate
\begin{equation}\label{eq:ABP-eigen-minus}
\sup_\Omega u^- \leq C_1 \left( 1 + ( \lambda^-_1(F,\Omega) - \lambda )^{-1} \right) \left( \sup_{\partial \Omega} u^- + \|f^- \|_{L^n(\Omega)} \right).
\end{equation}
\end{thm}

Our analysis in this paper relies crucially on the dependence of the right side of \EQ{ABP-eigen} on $\lambda^+_1(F,\Omega) - \lambda$. Because the proof given by Quaas and Sirakov in \cite{Quaas:2008} requires subtle modifications to achieve this dependence, and for the sake of completeness, we give a complete proof of \THM{ABP-eigen} in \APP{app}.

\begin{rem}\label{rem:comparison}
Notice that the operator $F(D^2u,Du,u,x) - \lambda u$ satisfies the comparison principle in $\Omega$ if and only if $\lambda < \lambda^+_1(F,\Omega)$. Indeed, suppose $u$ and $v$ satisfy
\begin{equation*}
\left\{ \begin{aligned}
& F(D^2u,Du,u,x) - \lambda u \leq F(D^2v,Dv,v,x) - \lambda v & \mbox{in} & \ \Omega, \\
& u \leq v & \mbox{on} & \ \partial \Omega.
\end{aligned}\right.
\end{equation*}
Set $w:=u-v$ and use \EQ{Fsuperlinearminus} to get
\begin{equation*}
F(D^2w,Dw,w,x) -\lambda w \leq 0 \quad \mbox{in} \ \Omega.
\end{equation*}
Now recall that $w \leq 0$ on $\partial \Omega$, and employ \THM{ABP-eigen} to conclude that $w\leq 0$ in $\Omega$ in the case that $\lambda < \lambda^+_1(F,\Omega)$. For $\lambda \geq \lambda^+_1(F,\Omega)$, the principal eigenfunction $\varphi^+_1$ is a witness to the failure of the comparison principle for the operator $F-\lambda$ in $\Omega$.
\end{rem}

\begin{prop}\label{prop:quantify-MP-lambda}
Suppose that $\lambda < \lambda^+_1(F,\Omega)$ and $f\in C(\Omega) \cap L^n(\Omega)$. Then there is a unique viscosity solution $\varphi^{\lambda,f}\in C(\bar{\Omega})$ of the Dirichlet problem
\begin{equation}\label{eq:quantify-MP-lambda}
\left\{ \begin{aligned}
& F(D^2\varphi^{\lambda,f},D\varphi^{\lambda,f},\varphi^{\lambda,f},x) =  \lambda \varphi^{\lambda,f} + f & \mbox{in} & \ \Omega, \\
& \varphi^{\lambda,f} = 0 & \mbox{on} & \ \partial \Omega.
\end{aligned}\right.
\end{equation}
Moreover, if $f \geq 0$ in $\Omega$ and $f \not \equiv 0$, then
\begin{equation}\label{eq:varphi-lambda-f-blowup}
\lim_{\lambda \nearrow \lambda^+_1(F,\Omega)} \sup_\Omega \left| \varphi^{\lambda,f} \right| = +\infty,
\end{equation}
and the normalized function $\tilde{\varphi}^{\lambda,f} :=\varphi^{\lambda,f} / \| \varphi^{\lambda,f} \|_{L^\infty(\Omega)}$ converges uniformly on $\bar{\Omega}$ to the positive principal eigenfunction $\efunpos$ of $F$ in $\Omega$, as $\lambda \to \lambda^+_1(F,\Omega)$.
\end{prop}
\begin{proof}
The existence of solutions can be obtained via a standard argument using the Perron method. Large multiples of the principal eigenfunctions provide sub- and supersolutions of \EQ{quantify-MP-lambda}, at least for $f$ vanishing near $\partial \Omega$. For more general $f$, we can approximate using \EQ{Fsuperlinearminus} and \THM{ABP-eigen}. Uniqueness follows at once from \REM{comparison}. See \cite{Armstrong:2009} or \cite{Quaas:2008} for a complete proof.

We now demonstrate \EQ{varphi-lambda-f-blowup} under the assumption that $f\geq 0$ and $f \not\equiv 0$. In this case we have $\varphi^{\lambda,f} \geq 0$ and $\varphi^{\lambda,f} \not\equiv 0$. Using homogeneity, we see that the function $\tilde{\varphi}^{\lambda,f}$ is a viscosity solution of
\begin{equation}\label{eq:varphi-lambda-blowup-1}
F(D^2\tilde{\varphi}^{\lambda,f},D\tilde{\varphi}^{\lambda,f},\tilde{\varphi}^{\lambda,f},x) = \lambda \tilde{\varphi}^{\lambda,f} + f/ \| \varphi^{\lambda,f} \|_{L^\infty(\Omega)} \inOmega.
\end{equation}
If \EQ{varphi-lambda-f-blowup} does not hold, then we may find a subsequence $\lambda_k \to \lambda^+_1(F,\Omega)$ and a number $0 < \eta < \infty$ such that
\begin{equation*}
\lim_{k\to \infty} \| \varphi^{\lambda_k,f} \|_{L^\infty(\Omega)} = \eta > 0.
\end{equation*}
Recall that $\| \tilde{\varphi}^{\lambda,f} \|_{L^\infty(\Omega)} = 1$. By using local $C^\alpha$ estimates available for fully nonlinear elliptic equations (c.f. \cite{Caffarelli:Book}), and taking a further subsequence, we may assume that there exists a function $\tilde{\varphi} \in C(\bar{\Omega})$ such that
\begin{equation*}
\tilde{\varphi}^{\lambda_k,f} \rightarrow \tilde{\varphi} \quad \mbox{locally uniformly on} \quad \bar{\Omega} \quad \mbox{as} \quad k \to \infty.
\end{equation*}
Passing to limits in \EQ{varphi-lambda-blowup-1}, we see that $\tilde{\varphi}$ is a viscosity solution of
\begin{equation*}
F(D^2\tilde{\varphi},D\tilde{\varphi},\tilde{\varphi},x) = \lambda^+_1(F,\Omega) \tilde{\varphi} + f/ \eta \inOmega.
\end{equation*}
Since $f/\eta \geq 0$, this contradicts \cite[Proposition 6.1]{Armstrong:2009}. We have established \EQ{varphi-lambda-f-blowup}. We may now pass to limits in \EQ{varphi-lambda-blowup-1} to deduce that $\tilde{\varphi} = \efunpos$, and that the whole sequence $\tilde{\varphi}^{\lambda,f}$ converges uniformly as $\lambda \to \lambda^+_1(F,\Omega)$ to $\varphi^+_1$.
\end{proof}

Arguing in a similar way as \PROP{quantify-MP-lambda}, we obtain:

\begin{prop}\label{prop:quantify-MP-lambda-minus}
Suppose that $\lambda < \lambda^-_1(F,\Omega)$ and $f\in C(\Omega) \cap L^n(\Omega)$ is such that $f \leq 0$. Then there is a unique nonpositive viscosity solution $\bar{\varphi}^{\lambda,f}\in C(\bar{\Omega})$ of the Dirichlet problem
\begin{equation}\label{eq:quantify-MP-lambda-minus}
\left\{ \begin{aligned}
& F(D^2\bar{\varphi}^{\lambda,f},D\bar{\varphi}^{\lambda,f},\bar{\varphi}^{\lambda,f},x) = \lambda \bar{\varphi}^{\lambda,f} + f & \mbox{in} & \ \Omega, \\
& \bar{\varphi}^{\lambda,f} = 0 & \mbox{on} & \ \partial \Omega.
\end{aligned}\right.
\end{equation}
Moreover, if $f \leq 0$ in $\Omega$ and $f \not \equiv 0$, then
\begin{equation}\label{eq:varphi-lambda-f-blowup-minue}
\lim_{\lambda \nearrow \lambda^-_1(F,\Omega)} \sup_\Omega \left| \bar{\varphi}^{\lambda,f} \right| = +\infty,
\end{equation}
and the normalized function $\bar{\varphi}^{\lambda,f} / \| \bar{\varphi}^{\lambda,f} \|_{L^\infty(\Omega)}$ converges uniformly on $\bar{\Omega}$ to the negative principal eigenfunction $\varphi^-_1$ of $F$ in $\Omega$, as $\lambda \to \lambda^-_1(F,\Omega)$.
\end{prop}

\begin{rem}
We will also require the following property of $\varphi^{\lambda,f}$ which is easily deduced from \EQ{Fsuperlinear} and \REM{comparison}. Namely, if $f,g,h \in C(\Omega) \cap L^n(\Omega)$ such that $h \leq f+g$, then
\begin{equation*}
\varphi^{\lambda,h} \leq \varphi^{\lambda,f} + \varphi^{\lambda,g} \quad \mbox{in} \quad \Omega.
\end{equation*}
for any $\lambda < \lambda^+_1(F,\Omega)$. In particular, if $f,g$ and $h$ are nonnegative, then
\begin{equation}\label{eq:trianglevarphi}
\| \varphi^{\lambda,h} \|_{L^\infty(\Omega)} \leq \| \varphi^{\lambda,f} \|_{L^\infty(\Omega)} + \| \varphi^{\lambda,g} \|_{L^\infty(\Omega)}.
\end{equation}
\end{rem}


\subsection{Sion's minimax theorem}

The proof of \THM{minimax-eigenvalue}, like the original in \cite{Donsker:1975} for linear operators, will employ a minimax theorem due to Sion \cite{Sion:1958}. For the convenience of the reader, we will now state this result.

\begin{definition}
Let $A$ and $B$ be sets. A function $f:A\times B \to \R$ is called \emph{convex-like in} $A$ if, for any $x,y\in A$ and $0\leq \alpha \leq 1$, there exists $z\in A$ such that
\begin{equation*}
f(z,b) \leq \alpha f(x,b) + (1-\alpha) f(y,b) \quad \mbox{for every} \  b\in B.
\end{equation*}
Similarly, $f$ is called \emph{concave-like in} $B$ if, for any $x,y\in B$ and $0\leq \alpha \leq 1$, there exists $z\in B$ such that
\begin{equation*}
f(a,z) \geq \alpha f(a,x) + (1-\alpha) f(a,y) \quad \mbox{for every}\ a\in A.
\end{equation*}
\end{definition}

\begin{thm}[See Sion \cite{Sion:1958}]\label{thm:minimax}
Suppose that $A$ is a compact topological space, $B$ is a set, and $f:A\times B \to \R$ is concave-like in $B$, and upper semi-continuous and convex-like in $A$. Then
\begin{equation}
\inf_{x\in A}\sup_{y\in B} f(x,y) = \sup_{y\in B}\inf_{x\in A} f(x,y).
\end{equation}
\end{thm}


\section{Proof of main results} \label{sec:mainresults}

In this section we will prove Theorems \ref{thm:minimax-eigenvalue} and \ref{thm:Dirichlet-at-resonance}.


\subsection{Proof of \THM{minimax-eigenvalue}}
Let $J:\probmeas\times \posfun \to \R$ denote the functional
\begin{equation}\label{eq:J-functional-plus}
J(\mu,\varphi) = \int_\Omega \frac{F(D^2\varphi(x),D\varphi(x),\varphi(x),x)}{\varphi(x)}\, d\mu(x).
\end{equation}
We will show now that $J$ satisfies the hypotheses of \THM{minimax}. First, recall that $\probmeas$ is a compact topological space (with respect to the usual weak-star topology), and for each fixed $\varphi \in \posfun$ the map
\begin{equation*}
\mu \mapsto J(\mu,\varphi)
\end{equation*}
is a continuous and linear on $\probmeas$. In particular, $J$ is continuous and convex-like in $\probmeas$. To see that $J$ is concave-like in $C^2_+(\bar{\Omega})$, select $u,v\in C^2_+(\bar{\Omega})$ and $0 < \alpha < 1$. A simple calculation reveals that for $w:=u^\alpha v^{1-\alpha}$ we have
\begin{align}
Dw & = w \left( \alpha \frac{Du}{u} + (1-\alpha) \frac{Dv}{v} \right),\\
\intertext{and}
D^2 w & = w\left( \alpha \frac{D^2u}{u} + (1-\alpha) \frac{D^2v}{v} - \alpha (1-\alpha)\left(\frac{Du}{u} - \frac{Dv}{v} \right) \otimes \left(\frac{Du}{u} - \frac{Dv}{v} \right) \right).
\end{align}
The operator $\puccisub{\gamma}{\Gamma}$ has the property that
\begin{equation*}
\puccisub{\gamma}{\Gamma}(-p \otimes p) = \gamma |p|^2
\end{equation*}
for any $p \in \R^n$. Using \HYP{Felliptic} and \EQ{Fsuperlinear} we immediately deduce
\begin{multline}\label{eq:key-convexity}
F(D^2w,Dw,w,x) \geq w \left( \frac{\alpha F(D^2u,Du,u,x)}{u} +  \frac{(1-\alpha)F(D^2v,Dv,v,x)}{v} \right.\\ 
\left. + \frac{\gamma \alpha (1-\alpha)}{u^2v^2} \left| vDu - uDv \right|^2 \right).
\end{multline}
Therefore
\begin{equation*}
J(\mu, w ) \geq \alpha J(\mu,u) + (1-\alpha) J(\mu,v)
\end{equation*}
for any $\mu \in \probmeas$, confirming that $J$ is concave-like in its second argument.

\medskip

Recall that for any fixed continous function $g \in C(\Omega)$,
\begin{equation*}
\inf_{x\in \Omega} g(x) = \inf_{\mu \in \probmeas} \int_\Omega g\, d\mu.
\end{equation*}
We now employ \EQ{maximin-formula-2} and \THM{minimax} to deduce that
\begin{equation}\label{eq:flipmaximin}
\begin{aligned}
\lambda^+_1(F,\Omega) & =  \sup_{\varphi\in \posfun} \inf_{x\in \Omega} \, \frac{F\left(D^2\varphi(x),D\varphi(x),\varphi(x),x\right)}{\varphi(x)} \\
& = \sup_{\varphi\in \posfun} \inf_{\mu \in \probmeas} J(\mu,\varphi) \\
& = \inf_{\mu \in \probmeas} \sup_{\varphi\in \posfun} J(\mu,\varphi). \qedhere
\end{aligned}
\end{equation}
Select a sequence $\{\mu_k\} \subseteq \probmeas$ for which
\begin{equation*}
\sup_{\varphi \in \posfun} J(\mu_k,\varphi) \to \lambda^+_1(F,\Omega). 
\end{equation*}
Up to a subsequence, there exists $\mu \in \probmeas$ such that $\mu_k \weakly \mu$ weakly in $\probmeas$, and it is immediate that
\begin{equation*}
\sup_{\varphi\in \posfun} J(\mu,\varphi) = \lambda^+_1(F,\Omega).
\end{equation*}
Thus the infimum in the last line of \EQ{flipmaximin} is actually a minimum. We have proven \EQ{minimax-1}.

\medskip

Denote the set of minimizing measures by
\begin{equation}\label{eq:minimizing-measures}
\dvmeas = \left\{ \mu \in \probmeas : \sup_{\varphi\in\posfun} J(\mu,\varphi) = \lambda^+_1(F,\Omega) \right\}.
\end{equation}
We have just seen that $\dvmeas$ is nonempty. From \EQ{eigenvalues-monotonic-in-domain} we see that $\mu(E) > 0$ for any $\mu \in \dvmeas$ and Borel set $E$ with $|E| > 0$.

\medskip

We claim that for any $\mu \in \dvmeas$ and any nonnegative $f\in C(\Omega)\cap L^n(\Omega)$,
\begin{equation}\label{eq:reciprocal-integrable}
\int_\Omega \frac{f}{\efunpos}\, d\mu \leq \liminf_{\lambda \nearrow \lambda^+_1(F,\Omega)} \left( \lambda^+_1(F,\Omega) - \lambda \right) \| \varphi^{\lambda,f} \|_{L^\infty(\Omega)}.
\end{equation}
We will first demonstrate \EQ{reciprocal-integrable} under the assumption that $f$ is smooth and positive on $\bar{\Omega}$. For each $\lambda < \lambda^+_1(F,\Omega)$, let $\varphi^{\lambda,f}$ and $\tilde{\varphi}^{\lambda,f}$ be as in \PROP{quantify-MP-lambda}. Employing $C^{2,\alpha}$ estimates we have $\varphi^{\lambda,f} \in C^2(\bar{\Omega})$, and $\varphi^{\lambda,f} \geq 0$ in $\Omega$. Using $\mu \in \dvmeas$ and \HYP{Felliptic} we obtain, for any $\ep >0$, 
\begin{equation*}
\lambda^+_1(F,\Omega) \geq J(\mu,\varphi^{\lambda,f}+\ep) \geq \int_\Omega \left( \frac{\lambda \varphi^{\lambda,f} - \consz \ep + f}{\varphi^{\lambda,f}+\ep}\right)\, d\mu.
\end{equation*}
Rearranging, we obtain
\begin{equation}\label{eq:reciprocal-integrable-4}
\int_\Omega \frac{-(\consz + \lambda)\ep+f}{\tilde{\varphi}^{\lambda,f} + \ep / \| \varphi^{\lambda,f} \|_{L^\infty(\Omega)}}\, d\mu \leq  \left( \lambda^+_1(F,\Omega) - \lambda \right) \| \varphi^{\lambda,f} \|_{L^\infty(\Omega)}.
\end{equation}
Since the integrand is positive for small enough $\ep > 0$, Fatou's Lemma allows us to send $\ep \to 0$ to get
\begin{equation*}
\int_\Omega \frac{f}{\tilde{\varphi}^{\lambda,f}}\, d\mu \leq \left( \lambda^+_1(F,\Omega) - \lambda \right) \| \varphi^{\lambda,f} \|_{L^\infty(\Omega)}.
\end{equation*}
Now we pass to the limit $\lambda \to \lambda^+_1(F,\Omega)$, using \PROP{quantify-MP-lambda} and Fatou's lemma again, to get \EQ{reciprocal-integrable} in the case that $f$ is positive and smooth.

For general nonnegative $f \in C(\Omega) \cap L^n(\Omega)$, select a sequence $\{ f_k \}$ of smooth, positive functions such that $f_k \to f$ pointwise and in $L^n(\Omega)$. Using \THM{ABP-eigen} and \EQ{trianglevarphi}, we see that
\begin{align*}
\int_\Omega \frac{f}{\varphi^+_1}\, d\mu & \leq  \liminf_{k\to \infty} \int_\Omega  \frac{f_k}{\varphi^+_1} \, d\mu \\
& \leq \liminf_{k\to \infty} \liminf_{\lambda \nearrow \lambda^+_1(F,\Omega)} \left( \lambda^+_1(F,\Omega) - \lambda \right) \| \varphi^{\lambda,f_k} \|_{L^\infty(\Omega)} \\
& \leq \liminf_{k\to \infty} \liminf_{\lambda \nearrow \lambda^+_1(F,\Omega)} \left( \lambda^+_1(F,\Omega) - \lambda \right) \left( \| \varphi^{\lambda,f} \|_{L^\infty(\Omega)} + \| \varphi^{\lambda,|f_k-f|} \|_{L^\infty(\Omega)} \right) \\
& \leq \liminf_{\lambda \nearrow \lambda^+_1(F,\Omega)} \left( \lambda^+_1(F,\Omega) - \lambda \right) \| \varphi^{\lambda,f} \|_{L^\infty(\Omega)}  + C_1  \liminf_{k\to \infty} \| f_k - f \|_{L^n(\Omega)} \\
& = \liminf_{\lambda \nearrow \lambda^+_1(F,\Omega)} \left( \lambda^+_1(F,\Omega) - \lambda \right) \| \varphi^{\lambda,f} \|_{L^\infty(\Omega)}. \qedhere
\end{align*}
This demonstrates \EQ{reciprocal-integrable}.

\medskip

According to \THM{ABP-eigen} and \EQ{reciprocal-integrable}, for any $\mu \in \dvmeas$ and $f \in C(\Omega) \cap L^n(\Omega)$, we have the estimate
\begin{equation*}
\left| \int_\Omega \frac{f}{\efunpos}\, d\mu \right| \leq \int_\Omega \frac{|f|}{\efunpos}\, d\mu \leq C_1 \| f \|_{L^n(\Omega)}.
\end{equation*}
Hence the linear functional
\begin{equation*}
f \mapsto \int_\Omega \frac{f}{\efunpos} \, d\mu
\end{equation*}
can be extended to a bounded linear functional on $L^n(\Omega)$, and there exists $\efunadjpos{\mu} \in L^{n/(n-1)}(\Omega)$ such that
\begin{equation*}
d\mu(x) = \varphi^+_1(x) \varphi^*_\mu(x) \, dx.
\end{equation*}
Recall that $\mu(E) > 0$ for any Borel set $E$ of positive Lebesgue measure. Thus the set $\{ \varphi^*_\mu = 0 \}$ is of zero Lebesgue measure. This completes the proof of \THM{minimax-eigenvalue}. \qed

\begin{rem}
The principal eigenvalue $\lambda^+_1(F,\Omega)$ also satisfies the relaxed minimax formula
\begin{equation}
\lambda^+_1(F,\Omega) = \min_{\psi} \sup_{\varphi} \| \psi  \|_{L^1(\Omega)}^{-1} \int_\Omega \frac{F(D^2\varphi,D\varphi,\varphi,x)}{\varphi}\, \psi\, dx
\end{equation}
where the minimum is taken over all positive $\psi\in L^{n/(n-1)}(\Omega)$ and the maximum over positive functions $\varphi\in W^{2,n}(\Omega) \cap C^0_+(\Omega)$.
\end{rem}


\subsection{Proof of \THM{Dirichlet-at-resonance}}

We will assume without loss of generality that
\begin{equation*}
\lambda^+_1(F,\Omega) = 0,
\end{equation*}
so that the Dirichlet problem \EQ{DP} reads
\begin{equation} \label{eq:DP-1}
\left\{ \begin{aligned}
& F(D^2u,Du,u,x) =  f & \mbox{in} & \ \Omega, \\
& u = 0 & \mbox{on} & \ \partial \Omega,
\end{aligned}\right.
\end{equation}
Define the set
\begin{multline}\label{eq:def-of-S}
\mathcal{S}(F,\Omega) = \left\{  f \in C(\Omega) \cap L^n(\Omega) : \mbox{there exists} \ \varphi \in W^{2,n}(\Omega) \ \mbox{such that} \ \varphi \geq 0 \ \mbox{in} \ \Omega \right. \\
\left. \mbox{and} \ F(D^2\varphi,D\varphi,\varphi,x) \geq f \ \mbox{in} \ \Omega \right\}.
\end{multline}
Owing to the superlinearity \EQ{Fsuperlinear} of $F$, it is immediate that $\mathcal{S}(F,\Omega)$ is a convex cone in $C(\Omega)$. For each $n\leq p \leq \infty$, let $\overline{\mathcal{S}}_p(F,\Omega)$ denote the $L^p(\Omega)$-closure of $\mathcal{S}(F,\Omega) \cap L^p(\Omega)$.

\smallskip

\THM{Dirichlet-at-resonance} follows from Propositions \ref{prop:Sfomega-and-solvability}, \ref{prop:Spcharacterization} and \ref{prop:DP-sufficient}, below.

\begin{prop} \label{prop:Sfomega-and-solvability}
Assume that $f \in C(\Omega) \cap L^n(\Omega)$. The Dirichlet problem \EQ{DP-1} is solvable provided that $f \in\mathcal{S}(F,\Omega)$ and $\lambda^-_1(F,\Omega) > 0$. On the other hand, if \EQ{DP-1} has a solution, then $f - \ep \in \mathcal{S}(F,\Omega)$ for any $\ep > 0$.
\end{prop}
\begin{proof}
Suppose $\lambda^-_1(F,\Omega) > 0$, and there exists a nonnegative function $\varphi \in W^{2,n}(\Omega)$ such that
\begin{equation}\label{eq:Sfomega-1}
F(D^2\varphi,D\varphi,\varphi,x) \geq f \quad \mbox{in} \quad \Omega.
\end{equation}
According to \PROP{quantify-MP-lambda-minus}, there exists a unique nonpositive solution $\psi \in W^{2,n}(\Omega)$ of the Dirichlet problem
\begin{equation*}
\left\{ \begin{aligned}
& F(D^2\psi,D\psi,\psi,x) = -|f| & \mbox{in} & \ \Omega, \\
& \psi = 0 & \mbox{on} & \ \partial \Omega.
\end{aligned}\right.
\end{equation*}
Since $\psi \leq 0 \leq \varphi$ on $\bar{\Omega}$, the existence of a solution $u \in W^{2,n}(\Omega)$ of \EQ{DP-1} follows from the standard Perron method for viscosity solutions and the $W^{2,p}$ estimates (see Winter \cite{Winter:2008}). 

On the other hand, suppose that $u$ is a solution of \EQ{DP-1}. According to the $W^{2,p}$ estimates, $u \in W^{2,n}(\Omega)$. Let $\ep > 0$, and select a compact subset $K \subseteq \Omega$ such that $u \geq -\ep$ on $\Omega \backslash K$. Select $k >0$ so large that $u + k\varphi^+_1 \geq 0$ on $K$. Set $\varphi := u + k\varphi^+_1 +\ep$. Then $\varphi \geq 0$ on $\bar{\Omega}$ and
\begin{equation*}
F(D^2\varphi,D\varphi,\varphi,x) \geq f - \consz \ep \quad \mbox{in} \quad \Omega.
\end{equation*}
Thus $f-\consz \ep \in \mathcal{S}(F,\Omega)$ for each $\ep >0$.
\end{proof}

\begin{rem}\label{rem:Sfomega-and-solvability}
If $f\in C(\Omega) \cap L^p(\Omega)$ for some $p> n$, then a solution $u$ of \EQ{DP-1} satisfies $u \in C^1(\bar{\Omega})$. Revising the argument above and using Hopf's Lemma, we see that $f\in \mathcal{S}(F,\Omega)$. 
\end{rem}

We will now characterize $\overline{\mathcal{S}}_p(F,\Omega)$ in terms of $\dvmeas$.

\begin{prop} \label{prop:Spcharacterization}
For each $n \leq p < \infty$, 
\begin{equation*}
\overline{\mathcal{S}}_p(F,\Omega) = \left\{ f \in L^p(\Omega) : \max_{\mu \in \dvmeas} \int_\Omega f \varphi^*_\mu\, dx \leq 0 \right\}.
\end{equation*}
Likewise,
\begin{equation*}
\overline{\mathcal{S}}_\infty(F,\Omega) = \left\{ f \in C(\bar{\Omega}) : \max_{\mu \in \dvmeas} \int_\Omega f \varphi^*_\mu \, dx \leq 0 \right\}.
\end{equation*}
\end{prop}
\begin{proof}
According to elementary Banach space theory, a closed convex cone is the intersection of the half-spaces that contain it. Thus we have
\begin{equation}\label{eq:Spcharacterization-duh}
\overline{\mathcal{S}}_p(F,\Omega) = \bigcap_{g \in E_p} \left\{ f \in L^p(\Omega) : \int_\Omega f g\, dx \leq 0 \right\}
\end{equation}
for each $n\leq p < \infty$, where we have defined
\begin{equation*}
E_p : = \left\{ g  \in L^{p/(p-1)}(\Omega) : \int_\Omega f g \, dx \leq 0 \ \mbox{for every} \ f\in \mathcal{S}(F,\Omega) \cap L^p(\Omega) \right\}.
\end{equation*}
Similarly, if we set
\begin{equation*}
E_\infty : = \left\{ \nu  \in \radonmeas : \int_\Omega f  \, d\nu \leq 0 \ \mbox{for every} \ f\in \mathcal{S}(F,\Omega) \cap C(\bar{\Omega}) \right\},
\end{equation*}
then
\begin{equation*}
\overline{\mathcal{S}}_\infty(F,\Omega) = \bigcap_{\nu \in E_\infty} \left\{ f \in C(\bar{\Omega}) : \int_\Omega f \, d\nu \leq 0 \right\}.
\end{equation*}
It is clear that $E_p \subseteq E_q$ for $n \leq p \leq q \leq \infty$, under the standard inclusion $L^p(\Omega) \subseteq \radonmeas$.

\smallskip

We claim that  
\begin{equation}\label{eq:Spcharacterization-wts}
E_\infty \subseteq \left\{ \nu \in \radonmeas : d\nu = c\varphi^*_\mu \,dx \ \mbox{for some} \ \mu\in \dvmeas, \ c\geq 0 \right\}.
\end{equation}
Select $\nu \in E_\infty$. Since $\mathcal{S}(F,\Omega)$ contains every nonpositive function in $C(\bar{\Omega})$, the measure $\nu \geq 0$. Assume that $\nu \not\equiv 0$. Define a probability measure $\mu\in \probmeas$ by 
\begin{equation*}
d\mu : = c^{-1} \varphi^+_1(x) \,d\nu, \quad c:= \int_\Omega \varphi^+_1 \, d\nu.
\end{equation*}
We will show that $\mu$ belongs to $\dvmeas$. Select a test function $\varphi \in C^2_+(\bar{\Omega})$. For each $0 < \alpha < 1$ and $\ep > 0$, define
\begin{equation*}
w_{\alpha,\ep} := \varphi^\alpha (\varphi^+_1 + \ep )^{1-\alpha},
\end{equation*}
and
\begin{equation*}
h_{\alpha,\ep}:= F(D^2w_{\alpha,\ep},Dw_{\alpha,\ep},w_{\alpha,\ep},x).
\end{equation*}
Then $w_{\alpha,\ep}\in C_+^2(\bar{\Omega})$, $h_{\alpha,\ep}\in C(\bar{\Omega})$, and by \EQ{key-convexity} we have
\begin{equation*}
h_{\alpha,\ep} \geq \left[ \alpha \frac{F(D^2\varphi,D\varphi,\varphi,x)}{\varphi} + (1-\alpha) \frac{F(D^2\varphi^+_1,D\varphi^+_1,\varphi^+_1+\ep ,x)}{\varphi^+_1+\ep} \right] w_{\alpha,\ep}.
\end{equation*}
Using $\nu \in E_\infty$ and \EQ{Fsuperlinear} we have
\begin{align*}
0 & \geq \frac{1}{\alpha} \int_\Omega h_{\alpha,\ep} \, d\nu \\
& \geq \int_\Omega \frac{F(D^2\varphi,D\varphi,\varphi,x)}{\varphi} w_{\alpha,\ep} \, d\nu + \frac{1-\alpha}{\alpha} \int_\Omega \frac{F(D^2\varphi^+_1,D\varphi^+_1,\varphi^+_1+\ep ,x)}{\varphi^+_1+\ep} w_{\alpha,\ep} \, d\nu \\
& \geq \int_\Omega \frac{F(D^2\varphi,D\varphi,\varphi,x)}{\varphi} w_{\alpha,\ep} \, d\nu + \frac{1-\alpha}{\alpha} \int_\Omega \lambda^+_1(F,\Omega) \frac{\varphi^+_1}{\varphi^+_1+\ep}w_{\alpha,\ep}\, d\nu \\
& \qquad \qquad - \frac{1-\alpha}{\alpha}\int_\Omega \frac{\consz \ep}{\varphi^+_1 + \ep} w_{\alpha,\ep} \, d\nu.
\end{align*}
Recall that $\lambda^+_1(F,\Omega) = 0$, and rearrange to obtain
\begin{align*}
\int_\Omega \frac{F(D^2\varphi,D\varphi,\varphi,x)}{\varphi} w_{\alpha,\ep} \, d\nu &  \leq  \frac{1-\alpha}{\alpha}\int_\Omega \frac{\consz \ep}{(\varphi^+_1 + \ep)^\alpha} \varphi^\alpha \, d\nu \\
& \leq \frac{1-\alpha}{\alpha}\int_\Omega \consz \ep^{1-\alpha} \varphi^\alpha \, d\nu.
\end{align*}
We may now send $\ep \to 0$ to deduce that
\begin{equation*}
\int_\Omega \frac{F(D^2\varphi,D\varphi,\varphi,x)}{\varphi} \varphi^\alpha (\varphi^+_1 )^{1-\alpha} \, d\nu \leq 0.
\end{equation*}
Now pass to the limit $\alpha \to 0$ to get
\begin{equation*}
\int_\Omega \frac{F(D^2\varphi,D\varphi,\varphi,x)}{\varphi}\, d\mu \leq 0.
\end{equation*}
It follows that $\mu \in \dvmeas$, as desired. We have demonstrated \EQ{Spcharacterization-wts}.

\smallskip

On the other hand, select that $\mu \in \dvmeas$ and $f\in \mathcal{S}(F,\Omega)$. We will show that
\begin{equation}\label{eq:S-characterization-2}
\int_\Omega f\varphi^*_\mu\, dx  \leq 0.
\end{equation}
Select $\varphi\in W^{2,n}(\Omega)$ such that $\varphi \geq 0$, and
\begin{equation*}
F(D^2\varphi,D\varphi,\varphi,x) \geq f \quad \mbox{in} \quad \Omega.
\end{equation*}
For $s > 0$, define $w_s:= \varphi^+_1 + s\varphi$, and notice that
\begin{align*}
0 & \geq \int_\Omega \frac{F(D^2w_s,Dw_s,w_s,x)}{w_s}\, d\mu \\
& \geq \int_\Omega \frac{F(D^2\varphi^+_1, D\varphi^+_1,\varphi^+_1,x) + sF(D^2\varphi,D\varphi,\varphi,x)}{\varphi^+_1 + s\varphi}\, d\mu\\
& \geq \int_\Omega \frac{sf}{\varphi^+_1 + s\varphi}\, d\mu.
\end{align*}
Thus
\begin{equation*}
\int_\Omega f \varphi^*_\mu \left( \frac{\varphi^+_1}{\varphi^+_1 + s\varphi } \right)\, dx \leq 0.
\end{equation*}
Now we may let $s\to 0$ to get \EQ{S-characterization-2}. We have shown that
\begin{equation*}
\left\{ c\varphi^*_\mu : \mu\in \dvmeas, \ c\geq 0 \right\} \subseteq E_n.
\end{equation*}
Combining with \EQ{Spcharacterization-wts}, we have that
\begin{equation*}
E_p = \left\{ c\varphi^*_\mu : \mu\in \dvmeas, \ c\geq 0 \right\}
\end{equation*}
for all $n\leq p < \infty$, as well as equality in \EQ{Spcharacterization-wts}. The result now follow from \EQ{Spcharacterization-duh}.
\end{proof}

The necessity of \EQ{necessary} for the solvability of the Dirichlet problem \EQ{DP} follows at once from Propositions \ref{prop:Sfomega-and-solvability} and \ref{prop:Spcharacterization}. The sufficiency of \EQ{plus-less-minus}  and \EQ{sufficient} is obtained from the next proposition.

\begin{prop}\label{prop:DP-sufficient}
Suppose that $\lambda^-_1(F,\Omega) > 0$, and $f\in C(\Omega) \cap L^n(\Omega)$ is such that
\begin{equation*}
\max_{\mu \in \dvmeas} \int_\Omega f \varphi^*_\mu \, dx < 0.
\end{equation*}
Then $f\in \mathcal{S}(F,\Omega)$. 
\end{prop}
\begin{proof}
Suppose first that $f\in C(\bar{\Omega})$. According to \PROP{Spcharacterization}, the function $f + \ep \in \overline{\mathcal{S}}_\infty(F,\Omega)$ for every sufficiently small $\ep >0$. Hence there exists $g \in \mathcal{S}(F,\Omega)$ such that $\| g - (f +\ep) \|_{L^\infty(\Omega)} \leq \ep$. In particular, $g \geq f$. Thus $f \in \mathcal{S}(F,\Omega)$.

For general $f\in C(\Omega) \cap L^n(\Omega)$, we may assume without of generality that $f$ is bounded below. Select $\alpha > 0$ small enough that
\begin{equation*}
\max_{\mu \in \dvmeas} \int_\Omega (f+\alpha) \varphi^*_\mu \, dx < 0.
\end{equation*}
For each $\ep > 0$, define
\begin{equation*}
f^\ep := f \wedge \ep^{-1}.
\end{equation*}
Then $f^\ep + \alpha \in \mathcal{S}(F,\Omega)$ for each $\ep > 0$. Fix $\ep > 0$ so small that
\begin{equation*}
\| f - f^\ep \|_{L^n(\Omega)} \leq \alpha / 2C_1,
\end{equation*}
where $C_1$ is as in \THM{ABP-eigen}. Consider the function $\varphi:= \varphi^{ -1 , f - f^\ep}$, which satisfies
\begin{equation*}
\left\{ \begin{aligned}
& F(D^2\varphi,D\varphi,\varphi,x) = - \varphi + f - f^\ep & \mbox{in} & \ \Omega, \\
& \varphi = 0 & \mbox{on} & \ \partial \Omega.
\end{aligned}\right.
\end{equation*}
According to \THM{ABP-eigen}, $0 \leq \varphi \leq 2C_1 \| f - f^\ep \|_{L^n(\Omega)} \leq \alpha$, and thus
\begin{equation*}
F(D^2\varphi,D\varphi,\varphi,x) \geq f - f^\ep -\alpha \quad \mbox{in} \quad \Omega.
\end{equation*}
Therefore, $f - f^\ep - \alpha \in  \mathcal{S}(F,\Omega)$. It follows that 
\begin{equation*}
f = (f^\ep + \alpha) + (f - f^\ep - \alpha) \in  \mathcal{S}(F,\Omega). \qedhere
\end{equation*}
\end{proof}

\begin{rem}
Suppose $u$ and $v$ are solutions of \EQ{DP} such that $u(\tilde{x}) > v(\tilde{x})$ at some point $\tilde{x}\in \Omega$. Using the concavity and homogeneity of $F$, we see that the function $w:=u-v$ satisfies
\begin{equation*}
F(D^2w,Dw,w,x) \leq \lambda^+_1(F,\Omega) w \quad \mbox{in} \ \Omega.
\end{equation*}
Comparing $w$ with $\varphi^+_1$ and applying \THM{BNV}, we deduce that $w\equiv t \varphi^+_1$ for some constant $t > 0$. Thus the difference of any two solutions of \EQ{DP} is a multiple of the principal eigenfunction, and in particular does not change sign in $\Omega$. According to the hypothesis \EQ{plus-less-minus} and \THM{ABP-eigen}, there exists a constant $C$ such that any solution of \EQ{DP} is bounded below by $-C$. Therefore, if \EQ{DP} is solvable, then it possesses a minimal solution.
\end{rem}

As an application of our techniques, we offer the following refinement of \EQ{reciprocal-integrable}:

\begin{prop}
Suppose $f \in C(\Omega) \cap L^n(\Omega)$ is nonnegative and let $\varphi^{\lambda,f}$ be as in \PROP{quantify-MP-lambda}. Then
\begin{equation}
\max_{\mu \in \dvmeas} \int_\Omega \frac{f}{\efunpos}\, d\mu = \lim_{\lambda \nearrow \lambda^+_1(F,\Omega)} \left( \lambda^+_1(F,\Omega) - \lambda \right) \| \varphi^{\lambda,f} \|_{L^\infty(\Omega)}.
\end{equation}
\end{prop}
\begin{proof}
Assume without loss of generality that $f \not \equiv 0$. Define
\begin{equation*}
k := \max_{\mu \in \dvmeas} \int_\Omega \frac{f}{\efunpos}\,d\mu
\end{equation*}
and
\begin{equation*}
l := \limsup_{\lambda \nearrow \lambda^+_1(F,\Omega)} \left( \lambda^+_1(F,\Omega) - \lambda \right) \| \varphi^{\lambda,f} \|_{L^\infty(\Omega)}.
\end{equation*}
Suppose on the contrary that $k < l$. Select $0 < \ep \leq (l-k)/2$. According to \PROP{DP-sufficient}, there exists a supersolution $u> 0$ of
\begin{equation*}
F(D^2u,Du,u,x) \geq \lambda^+_1(F,\Omega) u + f - (k+\ep) \efunpos + \delta \inOmega.
\end{equation*}
provided $\delta > 0$ is sufficiently small. We may write \EQ{quantify-MP-lambda} as
\begin{equation*}
F(D^2 \varphi^{\lambda,f},D\varphi^{\lambda,f},\varphi^{\lambda,f},x) = \lambda^+_1(F,\Omega) + f - \left(\lambda^+_1(F,\Omega) - \lambda\right) \varphi^{\lambda,f} \inOmega.
\end{equation*}
Take a subsequence $\lambda_j \to \lambda^+_1(F,\Omega)$ such that
\begin{equation*}
l = \lim_{j\to \infty} \left( \lambda^+_1(F,\Omega) - \lambda_j \right) \| \varphi^{\lambda_j,f} \|_{L^\infty(\Omega)}.
\end{equation*}
According to \PROP{quantify-MP-lambda},
\begin{equation}
\left(\lambda^+_1(F,\Omega) - \lambda_j\right) \varphi^{\lambda_j,f} \rightarrow l\efunpos \quad \mbox{uniformly in} \quad \Omega.
\end{equation}
Hence for $j$ sufficiently large, 
\begin{equation*}
\left(\lambda^+_1(F,\Omega) - \lambda_j\right) \varphi^{\lambda_j,f} \geq (k+\ep) \varphi^+_1 - \delta/2.
\end{equation*}
Thus the function $w:=\varphi^{\lambda_j,f} - u$ satisfies the inequality
\begin{equation*}
F(D^2w,Dw,w,x) -\lambda^+_1(F,\Omega) w \leq - \delta/2 \inOmega.
\end{equation*}
Since $w \leq 0$ on $\partial \Omega$, we may apply \THM{BNV} to deduce that $w\leq 0$ in $\Omega$. Hence for $j$ sufficiently large,
\begin{equation*}
\varphi^{\lambda_j,f} \leq u \inOmega,
\end{equation*}
in contradiction to \EQ{varphi-lambda-f-blowup}. Hence $l \leq k$.
\end{proof}


\section{Examples}

\label{sec:examples}

Theorems \ref{thm:minimax-eigenvalue} and \ref{thm:Dirichlet-at-resonance} suggest that, in analogy with linear operators, we should interpret the functions $\varphi^*_\mu$ as one-sided ``principal eigenfunctions of the adjoint of the linearization of $F$." As we will see below, the set $\dvmeas$ is not a singleton set in general. However, if the operator $F$ is differentiable in $(M,p,z)$ at the point $\left( D^2\varphi^+_1(x),D\varphi^+_1(x), \varphi^+_1(x),x\right)$ for every $x\in \Omega$, and this derivative is continuous, then there is only one minimizing measure $\mu \in \dvmeas$ and $\varphi^*_\mu$ is the principal eigenfunction of the adjoint of the linearization of $F$ about $\varphi^+_1$, as we will now show.

To state this result precisely, we now introduce the standard notion of weak solution for linear elliptic equations in double divergence form, following Bauman \cite{Bauman:1984}. Suppose $L$ is the linear, uniformly elliptic operator given by \EQ{linearoperator}, with bounded, measurable coefficients. If $f\in L^\infty_{\mathrm{loc}}(\Omega)$, then we say a function $v\in L^1_{\mathrm{loc}}(\Omega)$ is a \emph{weak solution} of the adjoint equation
\begin{equation}\label{eq:adjointeq}
L^*v = f\quad \mbox{in} \ \Omega,
\end{equation}
provided the following integral identity holds for every smooth function $\psi$ with compact support in $\Omega$:
\begin{equation*}
\int_\Omega v L\psi \, dx = \int_\Omega f \psi \, dx.
\end{equation*}
A weak solution $v$ of \EQ{adjointeq} is formally a solution of the double divergence form PDE
\begin{equation*}
-\left( a^{ij}(x) v \right)_{ij} - \left( b^i(x) v \right)_i + c(x) v = f \inOmega.
\end{equation*}

Weak solutions of \EQ{adjointeq} do not possess much regularity; in fact, they need not be locally bounded. Since the inverse of $L$ is a compact linear operator on $L^n(\Omega)$, the inverse of the adjoint of $L$ is a compact linear operator on $L^{n/(n-1)}(\Omega)$. Thus a weak solution $v$ of \EQ{adjointeq} must necessarily belong to $L^{n/(n-1)}_{\mathrm{loc}}(\Omega)$, and belongs to $L^{(n/(n-1)}(\Omega)$ provided that $f$ does. 

\begin{prop}\label{prop:Fdiffadj}
Suppose that at every point $x\in \Omega$, the operator $F$ is differentiable in $(M,p,z)$ at\begin{equation*}
\left(D^2\efunpos(x),D\efunpos(x),\efunpos(x),x\right),
\end{equation*}
and this derivative is continuous in $x$. Let $L$ be the linear elliptic operator
\begin{multline}
L u := F_{m_{ij}}\left(D^2\efunpos(x),D\efunpos(x),\efunpos(x),x\right) u_{ij} \\ + F_{p_i}\left(D^2\efunpos(x),D\efunpos(x),\efunpos(x),x\right) u_j + F_z \left(D^2\efunpos(x),D\efunpos(x),\efunpos(x),x\right) u.
\end{multline}
Then the set $\dvmeas$ consists of a unique measure $\mu$ characterized by the fact the that $\varphi^*_\mu$ is the unique (suitably normalized) weak solution of the adjoint equation
\begin{equation} \label{eq:adjoint-equation}
L^* \varphi^*_\mu = \lambda^+_1(F,\Omega) \varphi^*_\mu \quad \mbox{in} \ \Omega.
\end{equation}
\end{prop}
\begin{proof}
According to \HYP{Felliptic}, our operator $F$ is Lipschitz in $(M,p,z)$, uniformly in $x\in \Omega$. Thus the coefficients 
\begin{gather*}
a^{ij} (x) := -F_{m_{ij}}\left(D^2\efunpos(x),D\efunpos(x),\efunpos(x),x\right), \\
b^i(x) := F_{p_j}\left(D^2\efunpos(x),D\efunpos(x),\efunpos(x),x\right),
\intertext{and}
c(x) = F_{z}\left(D^2\efunpos(x),D\efunpos(x),\efunpos(x),x\right)
\end{gather*}
belong to $C(\bar{\Omega})$. Select a measure $\mu \in \dvmeas$ and a smooth function $\psi$ with compact support in $\Omega$. There exists $\delta > 0$ such that function $\psi^s:= \varphi^+_1 + s\psi > 0$ in $\Omega$ for $s\geq -\delta$. The map $s \mapsto J(\mu, \psi^s)$ has a local maximum at $s =0$. Thus we obtain
\begin{align*}
0 & = \left. \frac{\partial }{\partial s} J(\mu,\psi^s) \right|_{s=0} \\
& = \int_\Omega \frac{\partial}{\partial s} \left. \left( \frac{ F(D^2\psi^s,D\psi^s,\psi^s,x) }{\psi^s} \right) \right|_{s=0} \, d\mu \\
& = \int_\Omega \left[ \frac{-a^{ij}(x)\psi_{ij} + b^i(x) \psi_i + c(x) \psi}{\varphi^+_1} - \frac{ F(D^2\efunpos(x),D\efunpos(x),\efunpos(x),x)}{(\varphi^+_1)^2} \psi \right] \, d\mu \\
& = \int_\Omega  \left( L\psi - \lambda^+_1(F,\Omega) \psi \right) \varphi^*_\mu \, dx.
\end{align*}
Differentiating under the integral sign can be justified using the dominated convergence theorem and the fact that $F$ is Lipschitz in $(M,p,z)$, uniformly in $x$. It follows that $\varphi^*_\mu$ is a weak solution of \EQ{adjoint-equation}. According to the Fredholm alternative, $\lambda^+_1(F,\Omega)$ is a simple eigenvalue of the operator $L^*$. Therefore, $\dvmeas = \{ \mu \}$.
\end{proof}

\begin{exam}\label{exam:nonunique-measures}
We now demonstrate that $\dvmeas$ is not a singleton set in general. Observe that if $F_1$ and $F_2$ satisfy \HYP{Fcontinuous}-\HYP{Fconcave} and 
\begin{equation*}
F_1(M,p,z,x) - \lambda_1^+(F_1,\Omega)z \leq F_2(M,p,z,x) - \lambda_1^+(F_2,\Omega)z,
\end{equation*}
then we may immediately deduce that $\varphi_1^+(F_1,\Omega) \equiv \varphi^+_1(F_2,\Omega)$, and
\begin{equation*}
\dvmeascust{F_2}{\Omega} \subseteq \dvmeascust{F_1}{\Omega}.
\end{equation*}
From this we see that the maximum of two linear operators with proportional principal eigenfunctions but disproportional adjoint eigenfunctions should have two distinct minimizing measures.

For an explicit example, consider the operator
\begin{equation*}
G(D^2u,Du,u,x) = \min\{ -\Delta u,-\Delta u + \mathbf{b}(x)\cdot Du\} - |u|, \quad x\in B.
\end{equation*}
Here $B$ is the unit ball in $\R^n$, and we require $\mathbf{b}:B\to \R^n$ to be a smooth vector field satisfying
\begin{equation}\label{eq:vfb1}
x\cdot \mathbf{b}(x) = 0,\quad x\in B,
\end{equation}
and
\begin{equation}\label{eq:vfb2}
\divg \mathbf{b} \not\equiv 0.
\end{equation}
Let $\lambda_1$ and $\varphi_1$ denote the principal eigenvalue and eigenfunction of $-\Delta$ in $B$, respectively. Since $\varphi_1$ is radial, we have $\mathbf{b} \cdot D\varphi \equiv 0$. It follows that
\begin{equation*}
\lambda_1: = \lambda^+_1(G,B) + 1 = \lambda^-_1(G,B) - 1,
\end{equation*}
and $\varphi_1 \equiv \varphi^+_1(G,B) \equiv - \varphi^-_1(G,B)$. Let $\varphi_2$ denote the principal eigenfunction of the operator
\begin{equation*}
L^*_1u := -\Delta u - \mathbf{b}\cdot Du - (\divg \mathbf{b}) u,
\end{equation*}
which is the adjoint of the operator $L_1:= -\Delta u + \mathbf{b}\cdot Du$. Owing to \EQ{vfb1} and \EQ{vfb2}, we see that the functions $\varphi_1$ and $\varphi_2$ are not proportional.

The measures
\begin{equation*}
\mu_1 := \varphi_1^2\,dx / \| \varphi_1 \|_{L^2(\Omega)}^2 \quad \mbox{and} \quad \mu_2 := \varphi_1\varphi_2\, dx / \| \varphi_1 \varphi_2\|_{L^1(\Omega)}
\end{equation*}
belong to $\dvmeascust{G}{B}$. To see this, notice that $\mu_1\in \dvmeascust{-\Delta}{B}$ and $\mu_2\in \dvmeascust{L_1}{B}$, and apply the argument above. Alternatively, let $u\in C^2_+(\bar{\Omega})$ and check that
\begin{equation*}
\int_\Omega \frac{G(D^2u,Du,u,x)}{u}\,d\mu_1 \leq \int_\Omega \frac{-\Delta u - u}{u}\, d\mu_1 \leq \lambda_1 - 1 = \lambda^+_1(G,B).
\end{equation*}
Hence $\mu_1\in \dvmeascust{G}{B}$. Similarly, $\mu_2\in \dvmeascust{G}{B}$. 
\end{exam}

We conclude by demonstrating that in general the necessary condition \EQ{necessary} is not sufficient, nor is the sufficient condition \EQ{sufficient} necessary for the solvability of the boundary value problem \EQ{DP}. 

\begin{exam}\label{exam:nssn}
Consider the operator
\begin{equation*}
G(D^2u) := \min \left\{ -\Delta, -2\Delta \right\}
\end{equation*}
Observe that in any domain $\Omega$,
\begin{equation*}
\lambda_1:=\lambda^+_1(G,\Omega) = \lambda_1(-\Delta,\Omega) < 2\lambda_1(-\Delta,\Omega) = \lambda^-_1(G,\Omega)
\end{equation*}
and
\begin{equation*}
\varphi^+_1(G,\Omega) \equiv \varphi_1(-\Delta,\Omega) \equiv -\varphi^-_1(G,\Omega) =: \varphi_1.
\end{equation*}
According to \PROP{Fdiffadj},
\begin{equation*}
\dvmeascust{G}{\Omega} = \{ \mu \},
\end{equation*}
where $\mu$ is the probability measure given by
\begin{equation*}
d\mu(x) =  (\varphi_1(x) )^2\, dx / \| \varphi_1 \|_{L^2(\Omega)}^{2}.
\end{equation*}
Suppose that $f\in C^{1,\alpha}(\Omega)$ is such that
\begin{equation*}
\max_{\mu \in \dvmeascust{G}{\Omega}} \int_\Omega f\varphi^*_\mu = \int_\Omega f\varphi_1 \, dx = 0.
\end{equation*}
Then there is a solution $u$ of the Dirichlet problem
\begin{equation}\label{eq:examnssnD}
\left\{ \begin{aligned}
& -\Delta u =  \lambda_1 u + f & \mbox{in} & \ \Omega, \\
& u = 0 & \mbox{on} & \ \partial \Omega,
\end{aligned}\right.
\end{equation}
which is unique up to multiples of $\varphi_1$. We will argue that a solution $v$ of the Dirichlet problem
\begin{equation}\label{eq:examnssnG}
\left\{ \begin{aligned}
& G(D^2v) =  \lambda_1 v + f & \mbox{in} & \ \Omega, \\
& v = 0 & \mbox{on} & \ \partial \Omega,
\end{aligned}\right.
\end{equation}
exists if and only if
\begin{equation*}
f \geq 0 \quad \mbox{on} \ \partial \Omega.
\end{equation*}
If $v$ is a viscosity solution of \EQ{examnssnG}, then $v\in C^{2,\alpha}(\Omega)$. Define
\begin{equation*}
g: = -\Delta v -\lambda_1 v,
\end{equation*}
and notice that
\begin{equation*}
\int_\Omega g\varphi_1\, dx = 0
\end{equation*}
as well as
\begin{equation*}
g \geq G(D^2v) - \lambda_1 v = f \quad \mbox{in} \ \Omega.
\end{equation*}
It follows that $g\equiv f$ and $G(D^2v) = -\Delta v$. In particular, we conclude that $v$ is superharmonic in $\Omega$. Thus
\begin{equation*}
f = -\Delta v - \lambda_1 v = -\Delta v \geq 0 \quad \mbox{on} \ \partial \Omega.
\end{equation*}
On the other hand, suppose that $f \geq 0$ on $\partial \Omega$. Then a solution $u$ of \EQ{examnssnD} satisfies
\begin{equation*}
-\Delta u \geq 0 \quad \mbox{on} \ \partial \Omega.
\end{equation*}
Using Hopf's Lemma and $u\in C^3(\bar{\Omega})$, we may select $k \geq 0$ so large that the function $w_k:=u + k\varphi_1$ is superharmonic in $\Omega$. It follows that $w_k$ satisfies the boundary value problem \EQ{examnssnG}.
\end{exam}


\appendix

\section{Proof of \THM{ABP-eigen}} \label{app:app}

We assume that $F$ satisfies \HYP{Fcontinuous}, \HYP{Felliptic}, and \HYP{Fhomogeneous}. We do \emph{not} assume \HYP{Fconcave}. While this is not necessary for our purposes, the added generality will make our argument more clear in addition to providing a useful generalization of the result in \cite{Quaas:2008} to the case of non-concave $F$. Our proof will follow that of \cite{Berestycki:1994} for the linear case.

\begin{lem}\label{lem:eigenvalues-est-SD}
Suppose that $\Omega$ contains a ball $B$ of radius $R \leq 1$. Then we have the estimate
\begin{equation} \label{eq:eigenvalues-est-SD}
\consz + \lambda^+_1(F,\Omega) \leq \consEESD R^{-2},
\end{equation}
where $\consEESD$ depends only on $n$, $\gamma$, $\Gamma$, and $\consp$.
\end{lem}
\begin{proof}
We will suppose with loss of generality that $B$ is centered at $x=0$. Consider the auxiliary function $\varphi$ given by
\begin{equation*}
\varphi(x) := \frac{1}{4} \left( R^2 - |x|^2 \right)^2.
\end{equation*}
Performing a routine calculation, we see that
\begin{equation*}
D\varphi(x) = -\left( R^2 - |x|^2\right) x,\quad D^2\varphi(x) = -\left( R^2 - |x|^2\right) \imat + 2x\otimes x. 
\end{equation*}
In the ball $B$ we have
\begin{align*}
\Pucci(D^2\varphi) + \consp |D\varphi| & \leq n\Gamma \left( R^2 - |x|^2\right) - 2\gamma |x|^2 + \consp  \left( R^2 - |x|^2\right)|x| \\
& \leq 4\beta \varphi \left[ (n\Gamma + \consp) - 2\gamma |x|^2 \beta \right],
\end{align*}
where we have defined $\beta := (R^2 - |x|^2)^{-1}$. Let $0 < \alpha \leq \frac{1}{2}$ be a constant to be selected below. In the region $(1-\alpha) R^2 \leq |x|^2 < R \leq 1$, we have $\beta \geq \alpha^{-1} R^{-2}$, and
\begin{align*}
\Pucci(D^2\varphi) + \consp|D\varphi| & \leq 4 \beta \varphi \left[ (n\Gamma + \consp) -2\gamma|x|^2\beta \right] \\
& \leq 4\beta \varphi \left[ (n\Gamma + \consp) - 2\gamma \alpha^{-1} \right] \\
& \leq 0,
\end{align*}
provided that we choose 
\begin{equation*}
\alpha :=  \min \left\{ \textstyle{\frac{1}{2}}, 2\gamma(n\Gamma + \consp)^{-1} \right\}.
\end{equation*}
In the region $0 \leq |x| < (1-\alpha)R^2$, we have $\beta \leq \alpha^{-1} R^{-2}$, and thus
\begin{align*}
\Pucci(D^2\varphi(x)) + \consp|D\varphi(x) & \leq 4 \beta \varphi (n\Gamma + \consp) \\
& \leq 4 (n\Gamma + \consp) \alpha^{-1} R^{-2} \varphi \\
& = : \consEESD R^{-2} \varphi.
\end{align*}
In summary, we have shown that
\begin{equation*}
\Pucci(D^2\varphi) + \consp|D\varphi| \leq \consEESD R^{-2} \varphi \quad \mbox{in} \ B. 
\end{equation*}
Since $\varphi = 0$ on $\partial B$, it follows from \THM{BNV} that
\begin{equation*}
\lambda^+_1(\Pucci(D^2\cdot) + \consp|D\cdot|, B) \leq \consEESD R^{-2}.
\end{equation*}
Thus we have
\begin{equation*}
\consz + \lambda^+_1(F,\Omega) \leq \consz + \lambda^+_1(F,B) = \lambda^+_1(F + \consz,B) \leq \lambda^+_1(\Pucci(D^2\cdot) + \consp|D\cdot|, B) \leq \consEESD R^{-2},
\end{equation*}
as desired.
\end{proof}

\begin{lem}\label{lem:positive-supersolution}
For any $\lambda < \lambda^+_1(F,\Omega)$, there exists $v \in C^{1,\alpha}(\Omega)$ satisfying
\begin{equation}\label{eq:positive-supersolution-eq}
\left\{ \begin{aligned}
& F(D^2v,Dv,v,x) \geq \lambda v + 1 &  \mbox{in} & \ \Omega,\\
& v \geq 1 & \mbox{in} & \ \Omega,
\end{aligned} \right.
\end{equation}
as well as the estimate
\begin{equation}\label{eq:positive-supersolution-est}
\sup_\Omega v  \leq C \left( 1 + (\lambda^+_1(F,\Omega) - \lambda)^{-1} \right),
\end{equation}
where the constant $C$ depends only on $\Omega$, $n$, $\gamma$, $\Gamma$, and $\consp$. Moreover, we may take $v\in C^{1,\alpha}(\Omega)$, and if $F$ is concave \HYP{Fconcave}, then we may take $v\in C^{2,\alpha}(\Omega)$.
\end{lem}
\begin{proof}
According to \LEM{eigenvalues-est-SD}, there exists a constant $\eta >0$, depending only on $n$, $\gamma$, $\Gamma$, $\consp$, and the geometry of the domain $\Omega$, such that $\consz + \lambda^+_1(F,\Omega) \leq \eta$. Select smooth domains $\Omega_1 \subset \subset \Omega_2 \subset \subset \Omega_3 \subset \subset  \Omega$ such that
\begin{equation*}
| \Omega \backslash \Omega_1 | \leq \left( 2\consABP \eta \right)^{-n},
\end{equation*}
where $\consABP$ is the constant in the ABP inequality (see \cite[Proposition 2.12]{Caffarelli:1996}). Select a smooth function $g$ such that $0\leq g \leq 1$, $g\equiv 0$ on $\Omega_1$, and $g\equiv 1$ on $\Omega \backslash \Omega_2$. Let $u$ be the unique viscosity solution of the boundary value problem
\begin{equation}
\left\{ \begin{aligned}
& \pucci(D^2u) - \consp |Du| = g & \mbox{in} & \ \Omega, \\
& u = 0 & \mbox{on} & \ \partial \Omega.
\end{aligned} \right.
\end{equation}
Then $u \in C^{2,\alpha}(\Omega)$, and according to the ABP inequality,
\begin{equation*}
0\leq u \leq \consABP \| g \|_{L^n(\Omega)} \leq \consABP | \Omega \backslash \Omega_1|^{1/n} \leq \frac{1}{2\eta} \quad \mbox{in} \ \Omega.
\end{equation*}
Define $w:= 1+ \beta u$ for $\beta := 2(1+\eta)$. Then in the set $\Omega \backslash \Omega_2$, the function $w$ satisfies
\begin{align*}
\pucci(D^2w) - \consp|Dw| - (\lambda+ \consz)w \geq \beta  - (\lambda + \consz) \left( 1 + \frac{\beta}{2\eta}\right) \geq 1.
\end{align*}
According to the Harnack inequality (see \cite[Theorem 3.6]{Quaas:2008}), there exists $c>0$, depending on the appropriate constants and the geometry of $\Omega$, such that the positive principal eigenfunction $\varphi^+_1$ of $F$ satisfies
\begin{equation*}
\varphi^+_1 \geq c \quad \mbox{on} \ \bar{\Omega}_3
\end{equation*}
Define $v:=w + A \varphi^+_1$, where $A> 0$ will be selected below. In the set $\Omega \backslash \Omega_2$, we have
\begin{align*}
F(D^2v,Dv,v,x) - \lambda v & \geq A \left( F(D^2\varphi^+_1,D\varphi^+_1,\varphi^+_1,x) - \lambda\varphi^+_1\right) +  \pucci(D^2w) - \consp|Dw| - (\lambda+ \consz)w \\
& \geq 1.
\end{align*}
In the set $\Omega_3$, the function $v$ satisfies 
\begin{align*}
\lefteqn{F(D^2v,Dv,v,x) - \lambda v} \qquad & \\
& \geq A \left( F(D^2\varphi^+_1,D\varphi^+_1,\varphi^+_1,x) - \lambda\varphi^+_1\right) +  \pucci(D^2w) - \consp|Dw| - (\lambda+ \consz)w \\
& \geq A c (\lambda^+_1(F,\Omega) - \lambda ) - \eta \left( 1 + \frac{\beta}{2\eta} \right) \\
& = Ac (\lambda^+_1(F,\Omega) - \lambda ) +1 + 2\eta \\
& = 1,
\end{align*}
provided that we choose
\begin{equation*}
A:=\frac{2(1+\eta)}{c(\lambda^+_1(F,\Omega) - \lambda)}.
\end{equation*}
Therefore, the function $v$ satisfies \EQ{positive-supersolution-eq} and
\begin{equation*}
1 \leq v \leq 1 + \frac{1 + \eta}{\eta} + \frac{2(1+\eta)}{c(\lambda^+_1(F,\Omega)-\lambda)}.\qedhere
\end{equation*}
\end{proof}

\begin{proof}[Proof of \THM{ABP-eigen}]
Since we have not assumed that $F$ is concave, it suffices to show only the estimate \EQ{ABP-eigen}, since \EQ{ABP-eigen-minus} follows immediately by applying \EQ{ABP-eigen} to the operator $-F(-M,-p,-z,x)$. Let $w$ be the unique solution of the problem
\begin{equation}
\left\{ \begin{aligned}
& \Pucci(D^2w) + \consp |Dw| = - f^+ & \mbox{in} & \ \Omega, \\
& w = - u^+ & \mbox{on} & \ \partial \Omega.
\end{aligned} \right.
\end{equation}
Then $w \leq 0$, and according to the ABP inequality (\cite[Proposition 2.12]{Caffarelli:1996}) we have the estimate
\begin{equation*}
\| w \|_{L^\infty(\Omega)} \leq \sup_{\partial \Omega} u^+ + \consABP \| f^+ \|_{L^n(\Omega)}.
\end{equation*}
Define $z:= u + w$, and check that in the domain $\Omega$ the function $z$ satisfies
\begin{multline*}
F(D^2z,Dz,z,x) - \lambda z  \\
\leq F(D^2u,Du,u,x) -\lambda u + \Pucci(D^2w) + \consp|Dw| + \consz|w| + \lambda |w| \leq \eta |w|,
\end{multline*}
where $\eta$ is as in the proof of \LEM{positive-supersolution} above. Notice that $z \leq 0$ on $\partial \Omega$. According to \THM{BNV} and \EQ{positive-supersolution-eq},
\begin{equation*}
z \leq \left( \eta \| w \|_{L^\infty(\Omega)} \right) v,
\end{equation*}
where $v$ is as in \LEM{positive-supersolution}. Therefore,
\begin{align*}
\sup_\Omega u & \leq \sup_\Omega z + \| w \|_{L^\infty(\Omega)} \\
& \leq \left( 1 + \eta \| v \|_{L^\infty(\Omega)} \right) \| w \|_{L^\infty(\Omega)} \\
& \leq C \left( 1 + (\lambda^+_1(F,\Omega) - \lambda)^{-1} \right) \left( \sup_{\partial \Omega} u^+ + \consABP \| f^+ \|_{L^n(\Omega)} \right).\qedhere
\end{align*}
\end{proof}

\bibliographystyle{elsart-num-sort}
\bibliography{bibdb1}

\end{document}